\documentclass[12 pt]{article}
\title{Some Results on Local Distance Antimagic Chromatic Number of Graphs}
\author{Maurice Genevieva Almeida$^1$, Tarkeshwar Singh$^2$ \\
	$^1$p20230078@goa.bits-pilani.ac.in, $^2$tksingh@goa.bits-pilani.ac.in\\
	$^{1,}$ $^2$Birla Institute of Technology and Science Pilani,\\ K K Birla Goa Campus, Goa, India.}
\usepackage{amsthm, amssymb, amsfonts, amsmath, tikz}
\usepackage[affil-it]{authblk}
\usepackage{lineno}
\usepackage{url}
\usepackage[left=2.00cm, right=2.00cm, top=2.00cm, bottom=2.00cm]{geometry}
\newtheorem{theorem}{Theorem}[section]

\newtheorem{lemma}[theorem]{Lemma}

\newtheorem{proposition}[theorem]{Proposition}

\newtheorem{cor}[theorem]{Corollary}

\newtheorem{problem}[theorem]{Problem}

\usepackage{mathtools}
\newcommand{\ld}{\chi_{ld}}

\begin{document}
\date{}
\maketitle
\begin{abstract}
	Let $G=(V,E)$ be a graph of order $n$ without isolated vertices. A bijection $f\colon V\rightarrow \{1,2,\dots,n\}$ is called a local distance antimagic labeling, if $w(u)\not=w(v)$ for every edge $uv$ of $G$, where $w(u)=\sum_{x\in N(u)}f(x)$. The local distance antimagic chromatic number $\chi_{ld}(G)$ is defined to be the minimum number of colors taken over all colorings of $G$ induced by local distance antimagic labelings of $G$. In this paper, we study the local distance antimagic chromatic number for the join of graphs and the lexicographic product of graphs with the complement of the complete graph.
\end{abstract}
\section{Introduction}
 By a graph $G=(V, E)$, we mean a finite, simple, undirected graph having neither multiple edges nor loops. For graph theoretic notations, we refer to Chartrand and Lesniak \cite{chart}.\\ 
 
 The notion of antimagic labeling was introduced by Hartsfield and Ringel \cite{hartsfield} in 1990. A graph $G$ is antimagic if the edges of $G$ can be labeled by the numbers $\{1,2,\dots,|E|\}$ such that the sums of the labels of the edges incident to each vertex (called the weight of a vertex) are all distinct. They conjectured that {\it every connected graph with at least three vertices admits an antimagic labeling}. They also made a weaker conjecture that {\it every tree with at least three vertices admits an antimagic labeling}. These two conjectures were partly shown to be correct by several authors, but they are still unsolved.\\
 
Arumugam and Kamatchi \cite{Kamatchi} introduced a vertex version of antimagic labeling of a graph as follows:  a bijection $f: V \rightarrow \{1,2,\dots,n\}$ is said to be distance antimagic labeling of $G$ if all the vertices have distinct vertex weights, where the weight of a vertex is defined as $w(v)=\sum_{x\in N(v)}f(x)$, where $N(v)$ is the open neighborhood of the vertex $v$, which is defined as the set of vertices of the graph $G$ which are adjacent to $v$.  A graph $G$  is called a distance antimagic graph if it admits a distance antimagic labeling $f$.  For details (see \cite{Handathesis}, \cite{Handa}, \cite{Kamatchi}). \\


 Arumugam et al.\cite{premalatha} introduced a local version of antimagic labeling: let $G=(V,E)$ be a graph. A bijection $f\colon E\rightarrow \{1,2,\dots, |E|\}$ is called local antimagic labeling if for any two adjacent vertices $u$ and $v$, $w(u)\not=w(v)$, where $w(u)=\sum_{e\in E(u)}f(e)$ and $E(u)$ is the set of edges incident to $u$. Thus any local antimagic labeling induces a proper vertex coloring of $G$ where the vertex $v$ is assigned the color $w(v)$. The local antimagic chromatic number $\chi_{la}(G)$ is the minimum number of colors taken over all colorings induced by local antimagic labeling of $G$.\\

Arumugam et al.\cite{premalatha} conjectured that {\it a connected graph with at least three vertices admit a local antimagic labeling}. Bensmail et al. \cite{premalatha} solved this conjecture partially. Finally, Haslegrave proved this conjecture using probabilistic tools \cite{haslegrave}. Recently, several authors investigated the local antimagic chromatic number for several families of graphs. For further study, ( see \cite{premalatha}, \cite{lauLAG2}, \cite{raviLAG}, \cite{lauLAG1}).\\

Motivated by local antimagic labeling, Divya et al.\cite{yamini} and Handa et al.\cite{Handa} independently introduced the notion of local distance antimagic labeling as follows: let $G=(V, E)$ be a graph of order $n$ and let $f\colon V\rightarrow\{1,2,\dots,n\}$ be a bijection. For every vertex $v\in V$, define the weight of $v$ as $w(v)=\sum_{x\in N(u)}f(x)$. The labeling $f$ is said to be local distance antimagic labeling of $G$ if $w(u)\not=w(v)$ for every pair of adjacent vertices $u,v\in V$. A graph that admits such a labeling is called a local distance antimagic graph. A local distance antimagic labeling induces a proper vertex coloring of the graph, with the vertex $v$ assigned the color $w(v)$. The local distance antimagic chromatic number $\chi_{ld}(G)$ is the minimum number of colors taken over all colorings induced by local distance antimagic labelings of $G$. Clearly $\chi_{ld}(G) \geq \chi(G)$.\\ 
\section{Known Results}
Several authors have studied and found local distance antimagic chromatic numbers for different classes of graphs. For further study, (see  \cite{yamini}, \cite{Handa}, \cite{Nalliah1}, \cite{Nalliah2}, \cite{Nalliah3}, \cite{Nalliah4}). Handa et al.\cite{Handa} proved the following result, which is useful to get a lower bound for the local distance antimagic chromatic number of a graph.	
	\begin{proposition}\cite{Handa}\label{handaprop}
		Let $G$ be a local distance antimagic graph of order $n$. If $u$ and $v$ are vertices such that $|N(u)\triangle N(v)|=1\ or\ 2$, then $w(u)\not=w(v)$.
	\end{proposition} 
The following result by Priyadharshini et al.\cite{Nalliah1}, gives a lower bound for the local distance antimagic chromatic number of trees based on the number of support vertices it has.	
	\begin{theorem}\cite{Nalliah1}\label{nalliahleaf}
		Let $T$ be a tree on $n\geq 3 $ vertices with $k$ leaves and let $L=\{N(l), \text{where $l$ is a leaf}\}$. Also let $|L|=t$. Then $\chi_{ld}(T)\geq t+1$.
	\end{theorem}

	\begin{theorem}\cite{Nalliah4},\cite{Handa}\label{cycle}
		\begin{equation*}\chi_{ld}(C_n)=
			\begin{cases}
				2 &\text{n=4},\\
				3 &\text{n$\ \in \{3,12\}$},\\
				4 &\text{n$\ \in \{6,8,10,14\}$},\\
			 	5 &\text{n$\ \in \{5,7,9\}$}.
			\end{cases}
		\end{equation*}
		\begin{align*}
			4\leq \chi_{ld}(C_n)\leq 5;\ &n\in \{11,13\},\\
			4\leq \chi_{ld}(C_n)\leq 6;\ &n\geq 15.
		\end{align*}
  \end{theorem}
		\begin{theorem}\cite{Nalliah4},\cite{Handa}\label{path}
			\begin{equation*}
				\chi_{ld}(P_n)=
				\begin{cases}
					2 &\text{n$\ \in \{2,3\}$},\\
					3 &\text{n$\ \in \{5,11\}$},\\
					4 &\text{n$\ \in \{4,6,7,8,9,10\}$}.
				\end{cases}
				\end{equation*}
				\begin{equation*}
						4 \leq \chi_{ld}(P_n)\leq 
						\begin{cases}
							5 &\text{$n\geq 12$, $n$ is even},\\
							6 &\text{$n\geq 13$, $n$ is odd}.
						\end{cases}
				\end{equation*}
		
		\end{theorem}
  \begin{theorem}\cite{Handa}\label{multipartite}
		The complete multipartite graph $G=K_{n_1,n_2,\dots,n_r}$
		is local distance antimagic
		with $\chi_{ld}(G) = r$.
	\end{theorem}
 \begin{theorem}\cite{Handa}
     The wheel $W_n$, $n\geq 3$, is local distance antimagic with $3\leq \chi_{ld}(W_n) \leq 7$.
 \end{theorem}
 \begin{theorem}\cite{Handa}
     The complete graph $K_n$, $n\geq 2$, is local distance antimagic with $\chi_{ld}(K_n)=n$.
 \end{theorem}

 \begin{theorem}\cite{yamini}\cite{Handa}\label{friendship}
   The friendship graph $F_n$, $n\geq 2$, is local distance antimagic with $\chi_{ld}(F_n)=2n+1$.
 \end{theorem}
 
 \begin{theorem}\cite{yamini}\label{bistar}
   The bistar graph $B_{m,n}$, $m,n\geq 2$, is local distance antimagic with $\chi_{ld}(B_{m,n})=4.$
 \end{theorem}
	In this paper, we study the local distance antimagic chromatic number for the join of graphs and the lexicographic product of some classes of graphs with the complement of the complete graph and check if the results of the chromatic number of a graph hold good for local distance antimagic chromatic number also. We assume that all graphs $G$ considered in the paper admit a local distance labeling, with local distance antimagic chromatic number as $\chi_{ld}(G)$.    
	\section{The Main Results}

	\subsection{Some Results on Local Distance Antimagic Labeling of Graphs}
 Handa, in his thesis \cite{Handathesis}, studied the local distance antimagic labeling of graphs. He posted the following problem.
	\begin{problem}
		Given a fixed integer $p\in \{2,3,\hdots,n\}$, does there exist a graph $G$ of order $n$, with $\chi_{ld}(G)=p$?
  \end{problem}
  We provide two solutions to the problem.\\

\noindent{\textit{Solution:} }
			Consider the graph $G=K_{p-1}+\overline{K_{n-p+1}}$. We claim that $ \chi_{ld}(G)= p$.  Define a labeling $f$ for $G$, by labeling the $p-1$ vertices of $ K_{p-1}$, using integers from the set $\{1,2,3 \hdots , p-1\}$ in any order and then labeling the remaining $n-p+1$ vertices of $\overline{K_{n-p+1}}$, using integers from the set $\{p, p+1,\hdots,n\}$ in any order. Clearly, the weight of any vertex in $ \overline{K_{n-p+1}}$ is $\frac{(p-1)p}{2}$, while the weight of any vertex $v$ in $K_{p-1}$ is $\frac{(p-1)p}{2}- f(v)+\frac{(n-p+1)(n+p)}{2}$.
			All the $p-1$ vertices of  $ K_{p-1}$ receive distinct weights under $f$, while all the vertices in  $ \overline{K_{n-p+1}}$ receive the same weight, which is distinct from the weights of all the vertices of $ K_{p-1}$, as the maximum value that $f(v)$ can take is $p-1$. So the total number of distinct weights is p. Thus $\chi_{ld}(G)\leq p$.
			Now, 
			$\chi (G)= \chi (K_{p-1}+ \overline{K_{n-p+1}})=\chi(K_{p-1})+\chi( \overline{K_{n-p+1}}) = p-1+1 = p$, and 
			since $\chi_{ld}(G)\geq \chi(G)=p$, we get  $\chi_{ld}(G)=p$. \\
			We now present the second solution to the problem. For $p=2,3,\dots,n-1$, construct a $p$ partite graph $G$ with the first $p-1$ partite sets having $x_i$ elements, for $i=1,2,\dots,p-1$ where $x_i\geq 1$ and the last partite set having $n-\sum_{i=1}^{p-1}x_i$ elements, provided $n>\sum_{i=1}^{p-1}x_i$, that is, $G=K_{x_1,x_2,\dots,x_{p-1},n-\sum_{i=1}^{p-1}x_i}$. Clearly $G$ is graph of order $n$ such that $\chi_{ld}(G)=p$. For $p=n$, we know that $K_{n}$ is graph of order $n$, such that $\chi_{ld}(K_{n})=n$. \\
   
  In the first solution we see that, $\chi_{ld}(K_{n-k-1}+O_{k+1})=\chi_{ld}(K_{n-k-1})+\chi_{ld}(O_{k+1})$. Taking motivation from this result, in the next section, we study the join of two graphs and see for which graphs $G$ and $H$, $\chi_{ld}(G+H)=\chi_{ld}(G)+\chi_{ld}(H)$. 
		\subsection{$\chi_{ld}$ of Join of Graphs}
 If $G$ and $H$ are two graphs then $\chi(G+H)=\chi(G)+\chi(H)$  (see \cite{chart}).
  Further, using the definition of $\chi_{ld}$, we have an obvious lower bound for $\chi_{ld}(G+H)$, i.e., $\chi_{ld}(G+H) \geq \chi(G+H) =\chi(G)+\chi(H)$.\\

Our next goal is to provide an upper bound for the local distance antimagic chromatic number of the join of two graphs $G$ and $H$.

	\begin{theorem}\label{jointhm1}
		Let $G$ be a graph of order $n$ and $H$ be a graph of order $m$ such that $n\leq m$. Let $\Delta_{H}$ be the maximum degree of a vertex of $H$ and $\delta_{G}$ be the minimum degree of a vertex of $G$. If
		\begin{equation}\label{equation0}
			\Delta_{H}(m+n)-\frac{\Delta_{H}(\Delta_{H}-1)}{2}-\frac{\delta_{G}(\delta_{G}+1)}{2}<2nm+\frac{(m-n)(m+n+1)}{2},
		\end{equation} then $\chi_{ld}(G+H)\leq \chi_{ld}(G)+\chi_{ld}(H)$.
	\end{theorem}
	\begin{proof}
		Let $V(G)=\{v_{1},v_{2},\dots,v_{n}\}$ and $V(H)=\{x_{1},x_{2},\dots,x_{m}\}$ be the vertex sets $G$ and $H$ respectively. Let $f$ be the local distance antimagic labeling of $G$ that assigns $\chi_{ld}(G)$ distinct weights to vertices of $G$ and let $g$ be the local distance antimagic labeling of $H$ that assigns  $\chi_{ld}(H)$ distinct weights to vertices of $H$. Using the labelings $f$ and $g$, we define a new labeling $h$ for vertices of $G+H$ by
		$h(v_{i})=f(v_{i})$ and $h(x_{j})=g(x_{j})+n$, where $1\leq i \leq n$ and $1\leq j \leq m$.\\
		Under this labeling, the weights of vertices are, 
		\begin{align*}
			w_{G+H}(v_{i})&=w_{G}(v_{i})+\frac{m(m+1)}{2}+nm &\text{where $1\leq i \leq n$},\\
			w_{G+H}(x_{j})&=w_{H}(x_{j})+\frac{n(n+1)}{2}+n\cdot deg(x_{j}) &\text{where $1\leq j \leq m$}.
		\end{align*}
		Note that as adjacent vertices among $v_{i}$'s have distinct weights in $G$, they have distinct weights in $G+H$ too. Similarly, the adjacent vertices among $x_{i}$'s have distinct weights. Therefore the labeling $h$ induces $\chi_{ld}(G)$ distinct weights on vertices of $G$ and $\chi_{ld}(H)$ distinct weights on vertices of $H$ in $G+H$. Also as $n\leq m$, we have, $$\frac{n(n+1)}{2}+n\cdot deg(x_{j})<\frac{m(m+1)}{2}+nm.$$ 
		Note that the maximum possible value of $w_{G+H}(x_{j})$ is obtained when the value of $w_H(x_j)$ and $n \cdot \deg (x_j)$, is the maximum, i.e.,
		\begin{equation}\label{equation1}
			=m+(m-1)+\dots+(m-\Delta_{H}+1)+\frac{n(n+1)}{2}+n\cdot \Delta_{H}= \Delta_{H}(n+m)-\frac{\Delta_{H}(\Delta_{H}-1)}{2}+\frac{n(n+1)}{2},
		\end{equation}
		while the minimum possible value of $w_{G+H}(v_{i})$ is  obtained when the value of $w_G(v_i)$ is the minimum, i.e.,
		\begin{equation}\label{equation2}
			=1+2+\dots+\delta_{G}+\frac{m(m+1)}{2}+nm=\frac{\delta_{G}(\delta_{G}+1)}{2}+\frac{m(m+1)}{2}+nm.
		\end{equation}
		Now subtracting Equation \ref{equation1} from Equation \ref{equation2}, we have,
		\begin{align*}
			&=\frac{\delta_{G}(\delta_{G}+1)}{2}+\frac{m(m+1)}{2}+nm-\Delta_{H}(n+m)+\frac{\Delta_{H}(\Delta_{H}-1)}{2}-\frac{n(n+1)}{2}\\
			&=\frac{\delta_{G}(\delta_{G}+1)}{2}+\frac{\Delta_{H}(\Delta_{H}-1)}{2}-\Delta_{H}(m+n)+2nm+\frac{(m-n)(m+n+1)}{2}\\
			&>0 \quad \text{using Equation \ref{equation0}}.
		\end{align*} 
		Therefore, the weight of any vertex of $G$ in $G+H$ always exceeds the weight of any vertex of $H$ in $G+H$. Hence $h$ is a local distance antimagic labeling for $G+H$ and $\chi_{ld}(G+H)\leq \chi_{ld}(G)+\chi_{ld}(H)$.
	\end{proof}
 
	\begin{cor}\label{joincor}
		If $G$ is a graph of order $n$ and $H$ is a graph of order $m$, such that $n\leq m$ and $\Delta_{H}\leq n$, then $\chi_{ld}(G+H)\leq \chi_{ld}(G)+\chi_{ld}(H)$.
	\end{cor}
	
	\begin{cor}\label{joinwithempty1}
		If $n\leq m$ and $G$ is a graph of order $n$, then $\chi_{ld}(G+\overline{K_{m}})\leq \chi_{ld}(G)+1$.
	\end{cor}
	\begin{cor}\label{joinwithempty2}
		If $m<n$ and $G$ is a graph of order $n$ having maximum degree $\Delta$ and if $$\Delta(m+n)-\frac{\Delta(\Delta-1)}{2}<2nm+\frac{(n-m)(n+m+1)}{2},$$ then $\chi_{ld}(G+\overline{K_{m}})\leq \chi_{ld}(G)+1$.
	\end{cor}\bigskip
 
Theorem \ref{jointhm1}, gives a necessary condition of a local distance antimagic labeling of the join of two graphs $G$ and $H$ such that $\chi_{ld}(G+H)\leq\chi_{ld}(G)+\chi_{ld}(H)$. This bound is sharp. Equality holds for the graphs $G$ and $H$ satisfying  $\chi_{ld}(G)=\chi(G)$ and $\chi_{ld}(H)=\chi(H)$ respectively, i.e., $\chi_{ld}(G+H) \leq \chi_{ld}(G)+\chi_{ld}(H) \leq \chi(G)+\chi(H) = \chi(G+H)$. We know that $\chi_{ld}(G+H) \geq \chi(G+H)$, hence $\chi_{ld}(G+H) = \chi(G+H) = \chi(G) + \chi(H) = \chi_{ld} (G) + \chi_{ld}(H)$.\\

Next we present few more graphs $G$ and $H$, for which $\chi_{ld}(G+H)=\chi_{ld}(G)+\chi_{ld}(H)$.\\

	\begin{theorem}
		For positive integers $n$ and $m$, $\chi_{ld}(F_n+\overline{K_{m}})=2n+2$.
  \end{theorem}
  \begin{proof} 
  Let $V(F_n)=\{c, u_i, v_i:\ 1\leq i \leq n\}$ be the vertex set of $F_n$ and $V(\overline{K_{m}})=\{x_j:\ 1\leq j \leq m\}$ be the vertex set of $\overline{K_{m}}$ . The upper bound follows from Corollary \ref{joinwithempty1}, Corollary \ref{joinwithempty2}, and Theorem \ref{friendship}. We now turn to prove the lower bound.
  
      Consider any local distance antimagic labeling $g$ of $F_n+\overline{K_{m}}$. Since the vertices $c, u_i, v_i,\; 1\leq i \leq n$  form a clique, the weight of vertices $c, u_i, v_i,\; 1\leq i \leq n$ are distinct.  Now as $w(u_i)=g(v_i)+g(c)+\sum_{p=1}^{m}g(x_p)$, we have $w(u_i)\neq w(u_j)$, for $i\neq j$. Similarly for any $i\neq j$, we have $w(v_i)\neq w(v_j)$. Next, consider the weight of $u_i$ and $v_j$ for $i \neq j$. Suppose $w(u_i) = w(v_j)$ for any $i \neq j$ then we get $g(u_i)=g(v_j)$, which is a contradiction, therefore $w(u_i) \neq  w(v_j)$ for any $i \neq j$. Therefore the vertices of $F_n$ receive $2n+1$ distinct weights under $g$. Also, as all vertices of $F_n$ are adjacent to all vertices of $\overline{K_{m}}$, the weight of the vertices of $\overline{K_{m}}$ is distinct from these $2n+1$ weights. Therefore,  $\chi_{ld}(F_n+\overline{K_{m}})\geq 2n+2$. Hence $\chi_{ld}(F_n+\overline{K_{m}})= 2n+2$.
  \end{proof}
	
 \begin{theorem}
     For any positive integer $n$, $\chi_{ld}(F_n+B_{n,n})=2n+5$.
     \end{theorem}
     \begin{proof}
         Let $V(F_n)=\{c\}\cup \{u_i,v_i\ :\ 1\leq i\leq n\}$ be the vertex set of $F_n$ and $V(B_{n,n})=\{a,b\}\cup \{x_{i},y_i\ : 1\leq i \leq n\}$ be the vertex set of $B_{n,n}$. The upper bound follows from Corollary \ref{joincor}, Theorem \ref{friendship}, and Theorem \ref{bistar}. Our next aim is to prove that the lower bound for $\chi_{ld}(F_n+B_{n,n})$ is $2n+5$.\\         
         
         Consider any local distance antimagic labeling $g$ of $H=F_n+B_{n,n}$. Let $S=\sum_{v\in F_n}g(v)$ and $T=\sum_{v\in B_{n,n}}g(v)$. Consider the subgraph $B_{n,n}$ of $H$. Note that as $w(x_i)=g(a)+S$ and $w(y_i)=g(b)+S$, for $1\leq i \leq n$, all the $x_i's$ have the same weight, which is distinct from the weight of all the $y_i's$ as the labeling $g$ is a bijection. Note that, $w(a)=\sum_{i=1}^{n}g(x_i)+g(b)+S$ and $w(b)=\sum_{i=1}^{n}g(y_i)+g(a)+S$. For $1\leq i \leq n$ as the vertex $a$ and the vertices $x_i$ are adjacent $w(a)\not=w(x_i)$. Also $w(a)\not=w(y_i)$ as that would imply $\sum_{i=1}^{n}g(x_i)=0$. Similarly $w(b)$ is distinct from $w(x_i)$ and $w(y_i)$ for $1\leq i \leq n$. Also, as the vertices $a$ and $b$ are adjacent,  $w(a)\not =w(b)$. Therefore the vertices of subgraph $B_{n,n} $ receive 4 distinct weights under $f$. Now consider the subgraph $F_n$ in $H$. Since the vertices $c, u_i, v_i,\; 1\leq i \leq n$  form a clique, the weight of vertices $c, u_i, v_i,\; 1\leq i \leq n$ are distinct.  Now as $w(u_i)=g(v_i)+g(c)+T$, we have $w(u_i)\neq w(u_j)$, for $i\neq j$. Similarly for any $i\neq j$, we have $w(v_i)\neq w(v_j)$. Next, consider the weight of $u_i$ and $v_j$ for $i \neq j$. Suppose $w(u_i) = w(v_j)$ for any $i \neq j$ then we get $g(u_i)=g(v_j)$, which is a contradiction, therefore $w(u_i) \neq  w(v_j)$ for any $i \neq j$. Therefore the vertices of $F_n$ receive $2n+1$ distinct weights under $g$. Also, these $2n+1$ weights received by vertices of $F_n$ are distinct from the 4 weights received by vertices of $B_{n,n}$ under $g$, as all the vertices of $F_n$ are adjacent to all the vertices of $B_{n,n}$ in $H$ and $g$ is local distance antimagic labeling. So in total, the vertices of graph $H$ receive $2n+5$ distinct weights under $g$. As $g$ is arbitrary local distance antimagic labeling of $H$, we have $\chi_{ld}(H)\geq 2n+5$, and the equality follows.
     \end{proof}
	
	\subsection{$\chi_{ld}$  of Lexicographic Product of Graphs with Complement of Complete Graph}
		First, we define the lexicographic product of graphs. The lexicographic product of $G$ and $H$ is defined as a graph having vertex set $V(G)\times V(H)$, in which two vertices $(g,h)$ and $(g^{\prime},h^{\prime})$ are adjacent either if $g$ is adjacent to $g^{\prime}$ in $G$ or $g=g^{\prime}$ and $h$ is adjacent to $h^{\prime}$ in $H$.
	The lexicographic product of graphs, also called the composition of graphs, is denoted by $G[H]$. Next, we present a simple technique to construct the graph $G[H]$ from the graphs $G$ and $H$. Let $V(G)=\{x_{1},x_{2},\hdots, x_{n}\}$.  To construct $G[H]$,  corresponding to each vertex $x_{i}$ of $G$, we take a copy $H_{i}$ of $H$ for $i=1,2,\dots,n$, and if the vertices $x_{i}$ and $x_{j}$ are adjacent, then we join each vertex of $H_{i}$ to all the vertices of $H_{j}$.   Geller et al. \cite{lexichromatic}, proved that for a bipartite graph $G$ and for any graph $H$, $\chi(G[H])=2\ \chi(H)$. In this section, we study $\chi_{ld}(G[\overline{K_{n}}]) $ for various graph classes $G$ and obtain almost similar results as for $\chi(G[\overline{K_{n}}])$. \bigskip \\
 We begin by showing the upper bound for $\chi_{ld}(G[\overline{K_{n}}])$ under some conditions is $ \chi_{ld}(G)$.
 \begin{theorem}
    Let $n>1$ and $f$ be a local distance antimagic labeling of $G$ that induces $\chi_{ld}(G)$ distinct weights. Suppose for two adjacent vertices $u$ and $v$ we have, 
    \begin{equation}\label{G[O_n] eq 1}
        \frac{n(n+1)}{2}\deg(v)+(w_f(v)-deg(v))n^2 \not=\frac{n(n+1)}{2}\deg(u)+(w_f(u)-\deg(u))n^2, 
    \end{equation}
    then $\chi_{ld}(G[\overline{K_{n}}])\leq \chi_{ld}(G).$
\end{theorem}
\begin{proof}
    Let $V(G)=\{v_1,v_2,\dots, v_m\}$ and $V(\overline{K_{n}})=\{x_1,x_2,\dots,x_n\}$ be the vertex sets of $G$ and $\overline{K_{n}}$ respectively. For $1\leq j \leq m$ and $1\leq i \leq n$, let $v_i^j$ be the vertices of $G[\overline{K_{n}}]$ that corresponds with vertices $v_j$ of $G$. Using the labeling $f$, we define a vertex labeling $g$ of $G[\overline{K_{n}}]$ by
    $$g(v_i^j)=i+(f(v_j)-1)n.$$
    For the weight of the vertex $v_i^j$; $j=1,2,\dots,m$ and $i=1,2,\dots,n$ we obtain,
    \begin{align*}
        w_g(v_i^j)&=\sum_{v_p\in N_G(v_j)}\sum_{i=1}^{n}g(v_i^p)\\
        &=\frac{n(n+1)}{2}\deg(v_j)+(w(v_j)-\deg(v_j))n^2.
    \end{align*}
    As Equation \ref{G[O_n] eq 1} hold, the adjacent vertices of $G[\overline{K_{n}}]$ have distinct weights and hence $g$ is a local distance antimagic labeling of $G[\overline{K_{n}}]$ and moreover $\chi_{ld}(G[\overline{K_{n}}])\leq \chi_{ld}(G).$ 
\end{proof}
	Next, we present constructions of specific matrices, which we shall use repeatedly in our proofs in this section.
	We first construct a $n\times m$ matrix $A$ ($n$ even) as follows:\\
	For $i=1,2,\dots,n$, define,
	\begin{equation}\label{matrixA}
		A=(a_{i,j})=\begin{cases}
			(i-1)m+j &\text{for $i\equiv1\ (\bmod\ 2)$ and  $1\leq j\leq m$, }\\
			im+1-j &\text{for $i\equiv0\ (\bmod\  2)$ and  $1\leq j \leq m$. }
		\end{cases}
	\end{equation}
	Note that the sum of entries in any column of $A$ is equal to $\frac{n(nm+1)}{2}$. \bigskip\\
	We now present the construction of two $n\times m$ matrices $B$ and $C$  ($n $ odd) using the entries $1,2,3,\hdots,2mn$, such that the sum of entries in any column of $B$ is a fixed constant while the sum of entries in any column of $C$ is another constant.
	\begin{equation}\label{matrixB}
		B=(b_{i,j})=\begin{cases}
			j &\text{for $i=1$ and $1\leq j\leq m$,}\\
			j+m &\text{for $i=2$ and $1\leq j\leq m$,}\\
			4m-(2j-2) &\text{for $i=3$ and $1\leq j\leq m$,}\\
			6m+2j-1+2m(i-4) &\text{for $i\geq 4$, $i$ even and $1\leq j\leq m$,}\\
			2im-(2j-1) &\text{for $i\geq 5$, $i$ odd and $1\leq j\leq m$.}\\
		\end{cases}
	\end{equation}
	\begin{equation}\label{matrixC}
		C=(c_{i,j})=\begin{cases}
			4m+1-2j &\text{for $i=1$ and $1\leq j\leq m$,}\\
			4m+j &\text{for $i=2$ and $1\leq j\leq m$,}\\
			5m+j &\text{for $i=3$ and $1\leq j\leq m$,}\\
			6m+2j+2m(i-4) &\text{for $i\geq 4$, $i$ even and $1\leq j\leq m$,}\\
			2im-(2j-2) &\text{for $i\geq 5$, $i$ odd and $1\leq j\leq m$.}\\
		\end{cases}
	\end{equation}
	Note that the sum of entries in any column of $B$ is $mn^{2}-4m+2$ while the sum of the entries in any column of $C$ is $mn^{2}+4m+n-2$. 	Now if for some $n$ and $m$, we have \begin{align*}
		mn^{2}+4m+n-2&=mn^{2}-4m+2\\
		8m+n&=4\\
		n&=4-8m.
	\end{align*}
	This implies that $n$ is even, which is a contradiction. Hence, the column sums are always distinct.\\\\ Using the constructions discussed above, we prove results about the lexicographic product of some graphs with the complement of the complete graph.
In the following theorem, a $r$-regular bipartite graph of order $m$ is a graph $G$ with bi-partition \(\{A, B\}\), where \(A = \{u_1, u_2, \dots, u_s\}\) and \(B = \{v_{1},v_{2} \dots, v_{s}\}\), where $2s=m$. 
	\begin{theorem}\label{bipartitelexi1}
		For integers $n>1$, $m>1$ and a $r$-regular bipartite graph $G$ of order $m$ , we have  \(\ld(G[\overline{K_{n}}]) = 2\).
	\end{theorem}
	\begin{proof}
		For $1\leq j \leq s$ and $1\leq i \leq n$, let $u_{i}^j$ be the vertices of $G[\overline{K_{n}}]$ that correspond with vertices $u_{j}$ of $G$ and let $v_{i}^j$ be the vertices of $G[\overline{K_{n}}]$ that correspond with  vertices $v_{j}$ of $G$.
		We present the proof in two cases depending upon the parity of $n$.
  \begin{description}
      \item[Case 1:] $n$ is even. \par
      We construct a matrix $A=(a_{i,j})$ of size $n\times s$ using the definition of the Matrix \ref{matrixA}. We use this matrix to label $G[\overline{K_{n}}]$. Define $f\colon V(G[\overline{K_{n}}])\to \{1,2,\dots,mn\}$ by 
  \begin{align*}
      f(u_{i}^j)&=a_{i,j} \quad\text{for $1\leq j \leq s$ and $1 \leq i \leq n$,}\\
      f(v_{i}^j)&=a_{i,j}+ns \quad\text{for $1\leq j \leq s$ and $1 \leq i \leq n$.}
  \end{align*}

		For the weight of the vertices we have,
  \begin{align*}
      w(u_{i}^j)&=\sum_{v_j\in N(u_j)}\sum_{i=1}^nf(v_i^j)=\big(\tfrac{n(ns+1)}{2}+n^{2}s\big)r\quad \text{for $1\leq j \leq s$  and $1\leq i \leq n$,}\\
      w(v_{i}^j)&=\sum_{u_j\in N(v_j)}\sum_{i=1}^nf(u_i^j)=\big(\tfrac{n(ns+1)}{2}\big)r \quad \text{for $1\leq j \leq s$  and $1\leq i \leq n$.}
  \end{align*}
		Both these weights are distinct, and hence $f$ is a local distance antimagic labeling of $G[\overline{K_{n}}]$.
  \item[Case 2:] $n$ is odd.\par
   In this case, we construct two matrices $A=(a_{i,j})$ and $B=(b_{i,j})$ of size $n\times s$ using definition of Matrix \ref{matrixB} and Matrix \ref{matrixC}. We use these matrices to label $G[\overline{K_{n}}]$. Define $f\colon V(G[\overline{K_{n}}])\to \{1,2,\dots,nm\}$ by
   \begin{align*}
       f(u_{i}^j)&=a_{i,j} \quad \text{for $1\leq j\leq s$ and $1\leq i \leq n$},\\
       f(v_{i}^j)&=b_{i,j} \quad \text{for $1\leq j\leq s$ and $1\leq i \leq n$}.
   \end{align*}

		For the weight of the vertices we have,
  \begin{align*}
      w(u_{i}^j)&=\sum_{v_j\in N(u_j)}\sum_{i=1}^nf(v_i^j)=(sn^{2}+4s+n-2)r\quad \text{for $1\leq j \leq s$ and $1\leq i \leq n$,}\\
      w(v_{i}^j)&=\sum_{u_j\in N(v_j)}\sum_{i=1}^nf(u_i^j)=(sn^{2}-4s+2)r\quad \text{for $1\leq j \leq s$ and $1\leq i \leq n$. }
  \end{align*}
  \end{description}
  As seen above, both these weights are distinct, and hence $f$ is a local distance antimagic labeling of $G[\overline{K_{n}}]$. We see that, in either case, we get two distinct weights and therefore $\chi_{ld}(G[\overline{K_{n}}])\leq 2$. As $\chi(G[\overline{K_{n}}])=2$, we have, $\chi_{ld}(G[\overline{K_{n}}])= 2$. 
	\end{proof}

	\begin{cor}\label{cycleemptycomposition1}
		For integers \(n>1\) and \(m \ge 4\) even, \(\ld(C_m[\overline{K_{n}}]) = 2\).
	\end{cor}
	
	\begin{cor}
		For integers \(n>1\) and $p\geq 2$, \(\ld(K_{p,p}[\overline{K_{n}}]) = 2\).
	\end{cor}
 We have seen that for a regular bipartite graph $G$, $\chi_{ld}(G[\overline{K_{n}}])=2$. Our next goal is to investigate the same result for non-regular bipartite graphs. The following theorem states that the result is true if all the vertices in a partite set have the same degree. 
 \begin{lemma}\label{nonreglemma1}
     For integers $n\geq 2$ even and $m>1$ and a non-regular bipartite graph $G$ of order $m$ with all the vertices in a partite set having the same degree, we have $\chi_{ld}(G[\overline{K_{n}}])=2$. 
 \end{lemma}
 \begin{proof}
      Let $V(G)=A\cup B$, where $A=\{u_1,u_2,\dots,u_s\}$ and $B=\{v_1,v_2,\dots,v_r\}$ are the two partite sets, be the vertex set of $G$ with $s+r=m$. Also, let $V(\overline{K_{n}})=\{x_1,x_2,\dots,x_n\}$ be the vertex set of $\overline{K_{n}}$. Let $deg(u_p)=a$, for all $1\leq p \leq s$ and $deg(v_q)=b$, for all $1\leq q \leq r$. As $G$ is a non-regular graph, we have $a\not=b$. For $1\leq i \leq n$, let $u_{i}^j$ and $v_{i}^j$ be the vertices of $G[\overline{K_{n}}]$ that correspond with vertices $u_j$ for $1\leq j \leq s$ and $v_j$ for $1\leq j \leq r$ respectively of $G$.
       As $n$ is even, we construct a matrix $A=(a_{i,j})$ of size $n\times m$ using the definition of the Matrix \ref{matrixA}. We use this matrix to label $G[\overline{K_{n}}]$. Define $f\colon V(G[\overline{K_{n}}])\to \{1,2,\dots,mn\}$ by 
         \begin{align*}
             f(u_{i}^j)&=a_{i,j}\quad \text{for $1\leq i \leq n$ and $1\leq j \leq s$},\\
              f(v_{i}^j)&=a_{i,s+j}\quad \text{for $1\leq i \leq n$ and $1\leq j \leq r$}.
         \end{align*}
         For the weight of vertices we have, 
         \begin{align*}
             w(u_{i}^j)&=\sum_{v_j\in N(u_j)}\sum_{i=1}^nf(v_i^j)=\frac{an(nm+1)}{2}\quad \text{for  $1\leq i \leq n$ and $1\leq j \leq s$, }\\
             w(v_{i}^j)&=\sum_{u_j\in N(v_j)}\sum_{i=1}^nf(u_i^j)=\frac{bn(nm+1)}{2}\quad \text{for $1\leq i \leq n$ and $1\leq j \leq r$. }
         \end{align*}
         As $a\not=b$, the two weights obtained are distinct, and hence $f$ is a local distance antimagic labeling of $G[\overline{K_{n}}]$, and the result follows.
 \end{proof}

 \begin{lemma}\label{nonreglemma2}
     For $n\geq 3$ odd and $m>1$ odd and a non-regular bipartite graph $G$ of order $m$ with all the vertices in a partite set having the same degree, we have, $\chi_{ld}(G[\overline{K_{n}}])=2$.
     \begin{proof}
           Let $V(G)=A\cup B$, where $A=\{u_1,u_2,\dots,u_s\}$ and $B=\{v_1,v_2,\dots,v_r\}$ are the two partite sets, be the vertex set of $G$ with $s+r=m$. Also, let $V(\overline{K_{n}})=\{x_1,x_2,\dots,x_n\}$ be the vertex set of $\overline{K_{n}}$. Let $deg(u_p)=a$, for all $1\leq p \leq s$ and $deg(v_q)=b$, for all $1\leq q \leq r$. As $G$ is a non-regular graph, we have $a\not=b$. For $1\leq i \leq n$, let $u_{i}^j$ and $v_{i}^j$ be the vertices of $G[\overline{K_{n}}]$ that correspond with vertices $u_j$ for $1\leq j \leq s$ and $v_j$ for $1\leq j \leq r$ respectively of $G$. As $n$ and $m$ are odd, we have a magic rectangle $A$ of size $n\times m$ with entries $(a_{i,j})$, having the sum of entries in each column as $\frac{n(nm+1)}{2}$. We use this magic rectangle to label $G[\overline{K_{n}}]$. Define $f\colon V(G[\overline{K_{n}}])\to \{1,2,\dots,mn\}$ by 
         \begin{align*}
             f(u_{i}^j)&=a_{i,j}\quad \text{for $1\leq i \leq n$ and $1\leq j \leq s$,}\\
              f(v_{i}^j)&=a_{i,s+j}\quad \text{for $1\leq i \leq n$ and $1\leq j \leq r$.}
         \end{align*}
         For the weight of the vertices we have, 
         \begin{align*}
             w(u_{i}^j)&=\sum_{v_j\in N(u_j)}\sum_{i=1}^nf(v_i^j)=\frac{an(nm+1)}{2}\quad \text{for  $1\leq i \leq n$ and $1\leq j \leq s$, }\\
             w(v_{i}^j)&=\sum_{u_j\in N(v_j)}\sum_{i=1}^nf(u_i^j)=\frac{bn(nm+1)}{2}\quad \text{for $1\leq i \leq n$ and $1\leq j \leq r$. }
         \end{align*}
          As $a\not=b$, the two weights obtained are distinct, and hence $f$ is a local distance antimagic labeling of $G[\overline{K_{n}}]$, and the result follows.
         
     \end{proof}
 \end{lemma}
 \begin{theorem}\label{nonregcomplete}
     For integers $n>1$ and $a\not=b$, $\chi_{ld}(K_{a,b}[\overline{K_{n}}])=2$.
     \begin{proof}
         Let $V(K_{a,b})=\{u_1,u_2,\dots,u_a\}\cup\{v_1,v_2,\dots,v_b\}$ and $V(\overline{K_{n}})=\{x_1,x_2,\dots,x_n\}$ be the vertex sets of $K_{a,b}$ and $\overline{K_{n}}$ respectively. For $1\leq i \leq n$, let $u_{i}^j$ and $v_{i}^j$ be the vertices of $K_{a,b}[\overline{K_{n}}]$ that correspond with vertices $u_j$ for $1\leq j \leq a$ and $v_j$ for $1\leq j \leq b$ respectively of $K_{a,b}$. If $n$ is even, the result follows from Lemma \ref{nonreglemma1}. If $n$ and $a+b$ are odd, the result follows from Lemma \ref{nonreglemma2}. We shall now prove the result for the case; $n$ is odd and $a+b$ is even. Without loss of generality, let us assume $a<b$. Define $f\colon V(K_{a,b}[\overline{K_{n}}])\to \{1,2,\dots,(a+b)n\}$ by 
         \begin{align*}
             f(u_{i}^j)&=(i-1)a+j\quad \text{for $i=1,2,\dots,n$ and $j=1,2,\dots,a$,}\\
             f(v_{i}^j)&=an+(i-1)b+j\quad \text{for $i=1,2,\dots,n$ and $j=1,2,\dots,b$.}
         \end{align*}
         For the weight of the vertices we have, 
         \begin{align*}
             w(u_{i}^j)&=\sum_{v_j\in N(u_j)}\sum_{i=1}^nf(v_i^j)=\frac{bn(bn+1)}{2}+abn^2\hspace{0.7cm}  \text{for $1\leq i \leq n$ and $j=1,2,\dots,a$,}\\
             w(v_{i}^j)&=\sum_{u_j\in N(v_j)}\sum_{i=1}^nf(u_i^j)=\frac{an(an+1)}{2}\hspace{0.7cm}  \text{for $1\leq i \leq n$ and $j=1,2,\dots,b$.}
         \end{align*}
         Clearly, both the weights are distinct and therefore, $f$ is a local distance antimagic labeling. Hence, $\chi_{ld}(K_{a,b}[\overline{K_{n}}])=2$.
         
     \end{proof}
 \end{theorem}

	Above, we have studied several examples of non-regular bipartite graphs $G$ for which $\chi_{ld}(G[\overline{K_{n}}])=2$. However, there are some non-regular bipartite graphs for which the result is not true. Below we present few examples of such graphs. 
	\begin{theorem}\label{nonregbipartite}
		Let $G$ be a bipartite graph with two vertices $x$ and $y$  belonging to the same partite set, such as $N(x)\subset N(y)$. Then $\chi_{ld}(G[\overline{K_{n}}])\geq 3$.
	\end{theorem}
	\begin{proof}
		For $1\leq i \leq n$, let $x_{i}$ and $y_{i}$ be the vertices of $G[\overline{K_{n}}]$ that correspond with vertices $x$ and $y$ of $G$. As $N(x)\subset N(y)$,  by the construction of $G[\overline{K_{n}}]$, $N(x_i)\subset N(y_i)$ for all $1\leq i \leq n$. Consider the vertices $x_1$ and $y_1$ of $G[\overline{K_{n}}]$. If $\chi_{ld}(G[\overline{K_{n}}])=2$, then as $x$ and $y$ belong to the same partite set, we should have $w(x_{1})=w(y_{1})$. But as $N(x_1)\subset N(y_1)$, this is not possible. Therefore $\chi_{ld}(G[\overline{K_{n}}])\geq 3$.
	\end{proof}
 \begin{theorem}
     Let $G$ be a tree of order $m\geq 3$. Then, $\chi_{ld}(G[\overline{K_{n}}])=2$ if and only if $G$ is a star.
     \begin{proof}
         Suppose $G=K_{1,\ m-1}$, then from Theorem \ref{nonregcomplete},  $\chi_{ld}(G[\overline{K_{n}}])=2$. 
         Conversely suppose $G$ is a non-star graph with $\chi_{ld}(G[\overline{K_{n}}])=2$. Then, $G$ has at least two support vertices. Let $x$ and $y$ be adjacent support vertices of $G$. Consider a leaf $u$ attached to the vertex $x$. As vertices $u$ and $y$ are adjacent to vertex $x$, they belong to the same partite set, and in addition, $N(u)\subset N(y)$. Therefore by Theorem \ref{nonregbipartite}, $\chi_{ld}(\overline{K_{n}})\geq3$. This is a contradiction; hence, the graph $G$ must be a star. 
     \end{proof}
 \end{theorem}
	
	\begin{cor}
		For integers $n>1$ and  $m\geq 3$, $\chi_{ld}(P_{m}[\overline{K_{n}}])\geq 3$.
	\end{cor}
		\begin{theorem}\label{pathcomposition1}
		For $n> 1$ even and $m\geq 3$, $\chi_{ld}(P_{m}[\overline{K_{n}}])\leq \begin{cases}
			3 &\text{if $m$ is odd,}\\ 	4 &\text{if $m$ is even.}
		\end{cases}$
	\end{theorem}
	\begin{proof}
		Let $V(P_m)=\{v_1,v_2,\dots,v_m\}$ and $V(\overline{K_{n}})=\{x_1,x_2,\dots,x_n\}$ be the vertex sets of $P_m$ and $\overline{K_{n}}$ respectively. Let $x_{i}^j$ be the vertices of $P_{m}[\overline{K_{n}}]$, where $1\leq j \leq m$ and $1\leq i \leq n$. We shall deal with the proof in two cases depending upon the parity of $m$.\par 
  \begin{description}
      \item[Case 1:] $m$ is odd.\par 
      Let $m=2k+1$ for some positive integer $k$. Using the definition of Matrix \ref{matrixA}, we construct two matrices $A=(a_{i,j})$ and $B=(b_{i,j})$ of size $n\times (k+1)$ and $n\times k$ respectively. We use these two matrices to label $P_m[\overline{K_{n}}]$.
		Define $f\colon V(P_m[\overline{K_{n}}])\to \{1,2,\dots,mn\}$ by 
		\begin{equation*}
			f(x_{i}^j)=
			\begin{cases}
				a_{i,(\frac{j+1}{2})} &\text{for $j\equiv 1\ (\bmod\ 2)$ and $1\leq i \leq n$},\\
				b_{i,(\frac{j}{2})}+n(k+1) &\text{for $j\equiv 0\ (\bmod\ 2)$ and $1\leq i \leq n$}.
			\end{cases}
		\end{equation*}
		For the weight of the vertices, we have, for $1\leq i \leq n$,  
  \begin{align*}
      w(x_i^1)&=\sum_{i=1}^n f(x_i^2)= \frac{n^2k+n}{2}+n^2k+n^2=\sum_{i=1}^n f(x_i^{m-1})= w(x_i^m),\\
      w(x_{i}^j)&=\sum_{i=1}^n\big(f(x_i^{j-1})+f(x_i^{j+1})\big)=
			\begin{cases}
				n^2k+2n^2k\\+2n^2+n&\text{for $j\equiv 1\ (\bmod\ 2)$, $j\not=1,m,$}\\\\
				n^2k+n^2+n &\text{for $j\equiv 0\ (\bmod\ 2)$.}
			\end{cases}
  \end{align*}
		
		It can easily be verified that all three weights are distinct. As $w(x_i^j)\neq w(x_i^{j+1})$, for all $1\leq i \leq n$ and $1\leq j \leq  m-1$, $f$ is a local distance antimagic labeling of $P_m[\overline{K_{n}}]$ that assigns three distinct weights.
  \item[Case 2:] $m$ is even.\par
  Let $m=2k$ for some positive integer $k$. Using the definition of Matrix \ref{matrixA} we construct a matrix $A=(a_{i,j})$ of size $n\times k$ and use it to label $P_m[\overline{K_{n}}]$.
		Define $f\colon V(P_m[\overline{K_{n}}])\to \{1,2,\dots,mn\}$ by 
		\begin{equation*}
			f(x_{i}^j)=
			\begin{cases}
				a_{i,(\frac{j+1}{2})} &\text{for $j\equiv 1\ (\bmod\ 2)$ and $1\leq i \leq n$,}\\
				a_{i,(\frac{j}{2})}+nk &\text{for $j\equiv 0\ (\bmod\ 2)$ and $1\leq i \leq n$.}
			\end{cases}
		\end{equation*}
		For the weight of the vertices, we have, for $1\leq i \leq n$,
  \begin{align*}
    w(x_{i}^1)&=\sum_{i=1}^n f(x_i^2)=\frac{n^2k+n}{2}+n^2k,\\
    w(x_{i}^m)&=\sum_{i=1}^n f(x_i^{m-1})=\frac{n^2k+n}{2},\\
    w(x_i^j)&=\sum_{i=1}^n\big(f(x_i^{j-1})+f(x_i^{j+1})\big)=
    \begin{cases}
        n^2k+n+2n^2k&\text{for $j\equiv 1\ (\bmod\  2)$, $j\not=1$,}\\
				n^2k+n &\text{for $j\equiv 0\ (\bmod\ 2)$, $j\not=m$.}
    \end{cases}
  \end{align*}
		
		It can easily be verified that all four weights are distinct. As $w(x_i^j)\neq w(x_i^{j+1})$, for all $1\leq i \leq n$ and $1\leq j \leq m-1$, $f$ is a local distance antimagic labeling of $P_m[\overline{K_{n}}]$ that assigns four distinct weights.
  \end{description}	
	\end{proof}
	\begin{theorem}\label{pathcomposition2}
		For $n>1$ odd and $m\geq 3$, $\chi_{ld}(P_{m}[\overline{K_{n}}])\leq \begin{cases}
			4 &\text{for $m\equiv 1\  (\bmod\ 4)$,}\\	3 &\text{for $m\equiv 3\ (\bmod\ 4)$,}\\ 	4 &\text{for $m$ is even.}
		\end{cases}$
	\end{theorem}
	\begin{proof}
		Let $V(P_m)=\{v_1,v_2,\dots,v_m\}$ and $V(\overline{K_{n}})=\{x_1,x_2,\dots,x_n\}$ be the vertex sets of $P_m$ and $\overline{K_{n}}$ respectively.  Let $x_{i}^j$ be the vertices of $P_{m}[\overline{K_{n}}]$, where $1\leq j \leq m$ and $1\leq i \leq n$. We shall deal with the proof in three cases depending upon the nature of $m$.\par 
  \begin{description}
      \item[Case 1:] $m\equiv 1\ (\bmod\ 4).$\par
      Let $m=4k+1$ for some positive integer $k$. Since $n$ and $2k+1$ are odd, there is a magic rectangle $A$ with entries $a_{i,j}$ of size $n\times (2k+1)$. Using the definition of Matrix \ref{matrixB} and Matrix \ref{matrixC}, we construct two matrices $B=(b_{i,j})$ and $C=(c_{i,j})$ respectively of size $n\times k$. We use these constructions to label $P_m[\overline{K_{n}}]$.
		Define $f\colon V(P_m[\overline{K_{n}}])\to \{1,2,\dots,mn\}$ by 
		\begin{equation*}
			f(x_{i}^j)=
			\begin{cases}
				a_{i,(\frac{j+1}{2})}+2nk &\text{for $j\equiv 1\ (\bmod\ 2)$ and $1\leq i \leq n$,}\\
				b_{i,(\frac{j+2}{2})} &\text{for $j\equiv 2\ (\bmod\ 4)$ and $1\leq i \leq n$,}\\
				c_{i,(\frac{j}{4})} &\text{for $j\equiv 0\ (\bmod\ 4)$ and $1\leq i \leq n$.}
			\end{cases}
		\end{equation*}
  For the weight of the vertices we have, for $1\leq i \leq n$,
  \begin{align*}
    w(x_{i}^1)&=\sum_{i=1}^n f(x_i^2)=kn^2-4k+2,\\
    w(x_{i}^m)&=\sum_{i=1}^n f(x_i^{m-1})=kn^2+4k+n-2 ,\\
    w(x_i^j)&=\sum_{i=1}^n\big(f(x_i^{j-1})+f(x_i^{j+1})\big)=
    \begin{cases}
        2n^2k+n &\text{for $j\equiv 1\ (\bmod\  2)$, $j\not=1$,}\\
	 n^2(2k+1)+n+4n^2k &\text{for $j\equiv 0\ (\bmod\ 2)$, $j\not=m$.}
    \end{cases}
  \end{align*}
		It can easily be verified that all four weights are distinct. As $w(x_i^j)\neq w(x_i^{j+1})$, for all $1\leq i \leq n$ and $1\leq j \leq m-1$, $f$ is a local distance antimagic labeling of $P_m[\overline{K_{n}}]$ that assigns four distinct weights.
  \item[Case 2:] $m\equiv 3\ (\bmod\ 4)$.\par
  Let $m=4k+3$ for some positive integer $k$. Since $n$ and $2k+1$ are odd there is a magic rectangle $A$ with entries $a_{i,j}$ of size $n\times (2k+1)$. Using the definition of Matrix \ref{matrixB} and Matrix \ref{matrixC}, we construct two matrices $B=(b_{i,j})$ and $C=(c_{i,j})$ respectively of size $n\times (k+1)$. We use these constructions to label $P_m[\overline{K_{n}}]$.
		Define $f\colon V(P_m[\overline{K_{n}}])\to \{1,2,\dots,mn\}$ by 
		\begin{equation*}
			f(x_{i}^j)=
			\begin{cases}
				a_{i,(\frac{j}{2})}+2n(k+1) &\text{for $j\equiv 0\ (\bmod\ 2)$ and $1\leq i \leq n$,}\\
				b_{i,(\frac{j+3}{4})} &\text{for $j\equiv 1\ (\bmod\ 4)$ and $1\leq i \leq n$,}\\
				c_{i,(\frac{j+1}{4})} &\text{for $j\equiv 3\ (\bmod\ 4)$ and $1\leq i \leq n$.}
			\end{cases}
		\end{equation*}
		For the weight of the vertices, we have, for $1\leq i \leq n$,
  \begin{align*}
      w(x_i^1)&=\sum_{i=1}^n f(x_i^2)=\frac{n^2(2k+1)+n}{2}+2n^2(k+1) =\sum_{i=1}^n f(x_i^{m-1})= w(x_i^m),\\
      w(x_{i}^j)&=\sum_{i=1}^n\big(f(x_i^{j-1})+f(x_i^{j+1})\big)=
			\begin{cases}
				n^2(6k+5)+n&\text{for $j\equiv 1\ (\bmod\ 2)$, $j\not=1,m,$}\\
				n^2(2k+2)+n  &\text{for $j\equiv 0\ (\bmod\ 2)$.}
			\end{cases}
  \end{align*}
		It can easily be verified that all three weights are distinct. As $w(x_i^j)\neq w(x_i^{j+1})$, for all $1\leq i \leq n$ and $1\leq j \leq m-1$, $f$ is a local distance antimagic labeling of $P_m[\overline{K_{n}}]$ that assigns three distinct weights.
  \item[Case 3:] $m$ is even.\par
  Let $m=2k$ for some positive integer $k$. Using the definition of Matrix \ref{matrixB} and Matrix \ref{matrixC} we construct matrices $B=(b_{i,j})$ and $C=(c_{i,j})$ respectively of size $n\times k$ and use them to label $P_m[\overline{K_{n}}]$.
		Define $f\colon V(P_m[\overline{K_{n}}])\to \{1,2,\dots,mn\}$ by 
		\begin{equation*}
			f(x_{i}^j)=
			\begin{cases}
				b_{i,(\frac{j+1}{2})} &\text{for $j\equiv 1\ (\bmod\ 2)$ and $1\leq i \leq n$,}\\
				c_{i,(\frac{j}{2})} &\text{for $j\equiv 0\ (\bmod\ 2)$ and $1\leq i \leq n$.}
			\end{cases}
		\end{equation*}
   For the weight of the vertices, we have, for $1\leq i \leq n$,
  \begin{align*}
    w(x_{i}^1)&=\sum_{i=1}^n f(x_i^2)=kn^2+4k+n-2,\\
    w(x_{i}^m)&=\sum_{i=1}^n f(x_i^{m-1})=kn^2-4k+2 ,\\
    w(x_i^j)&=\sum_{i=1}^n\big(f(x_i^{j-1})+f(x_i^{j+1})\big)=
    \begin{cases}
        2(kn^2+4k+n-2) &\text{for $j\equiv 1\ (\bmod\  2)$, $j\not=1$,}\\
	 2(kn^2-4k+2) &\text{for $j\equiv 0\ (\bmod\ 2)$, $j\not=m$.}
    \end{cases}
  \end{align*}
		It can easily be verified that all four weights are distinct. As $w(x_i^j)\neq w(x_i^{j+1})$, for all $1\leq i \leq n$ and $1\leq j \leq m-1$, $f$ is a local distance antimagic labeling of $P_m[\overline{K_{n}}]$ that assigns four distinct weights.
  \end{description}
	\end{proof}
	After studying the lexicographic product of bipartite graphs with the complement of the complete graph, we now study the lexicographic product of non-bipartite graphs with the complement of the complete graph. We begin by exploring some regular graphs with chromatic number 3, for which $\chi_{ld}(G[\overline{K_{n}}])=3$.

	\begin{theorem}
		Let \(n\) be even and \(G\) be a \(2r\)-regular graph of order \(m\) with chromatic number \(3\) such that for each vertex \(v\), \(N(v)\) can be partitioned into two monochromatic subsets each of size \(r\). Then \(\ld(G[\overline{K_{n}}]) = 3\).
	\end{theorem}
	\begin{proof}
		Let \( V(G) = \{v_1, \dots, v_m\}\)  be the vertex set of $G$, where  \(\{v_1, \dots, v_k\}\), \(\{v_{k+1}, \dots, v_{k+s}\}\), \(\{v_{k+s+1}, \dots, v_{k+s+t}\}\) are the color partitions with $k+s+t=m$. Note that each of \(k, s\) and \(t\) is atleast \( r\). Further, let $V(\overline{K_{n}})=\{x_{1},x_{2},\dots, x_{n}\}$ be the vertex set for $\overline{K_{n}}$. For $1\leq j \leq m$ and $1\leq i \leq n$, let $v_{i}^j$ be the vertices of $G[\overline{K_{n}}]$ that correspond with vertices $v_{j}$ of $G$. Using the definition of Matrix \ref{matrixA} above, we define three matrices $B=(b_{i,j})$, $C=(c_{i,j})$ and $D=(d_{i,j})$ of sizes $n\times k$, $n\times s$ and $n\times t$ respectively. 
		We use these three matrices to define a labeling $f\colon V(G[\overline{K_{n}}])\to \{1,2,\dots,mn\}$ as follows:
		\begin{equation*} f(v_{i}^j)=
			\begin{cases}
				b_{i,j} &\text{for $1\leq i \leq n$ and $1\leq j \leq k$,}\\
				c_{i,j}+nk &\text{for $1\leq i \leq n$ and $k+1\leq j \leq k+s$,}\\
				d_{i,j}+n(k+s) &\text{for $1\leq i \leq n$ and $k+s+1\leq j \leq m$.}
				
			\end{cases}
		\end{equation*}
		For the weight of vertices, we have, for $1\leq i \leq n$,
  \begin{align*}
     w(v_i^j)= &\sum_{v_p\in N(v_j)}\sum_{i=1}^{n}f(v_i^p)=
     \begin{cases}
         \frac{rn(ns+1)}{2} + \frac{rn(nt + 1)}{2} + rn^2(2k+s)\quad  \text{for $1\leq j \leq k$},\\\\
         \frac{rn(ns+1)}{2} + \frac{rn(nt + 1)}{2} + rn^2(k+s)\quad  \text{for $k+1\leq j \leq k+s$},\\\\
         \frac{rn(nk+1)}{2} + \frac{rn(ns + 1)}{2} + rn^2k\quad  \text{for $k+s+1\leq j \leq m$}.
     \end{cases}
    \end{align*}
		All three weights are distinct and therefore, $f$ is a local distance antimagic labeling of $G[\overline{K_{n}}]$. Hence $\chi_{ld}(G[\overline{K_{n}}])\leq 3$. As $\chi (G[\overline{K_{n}}])=3$, the equality follows.
	\end{proof}
	
	\begin{cor}
		For \(n\) even, \(\ld(K_{m,m,m}[\overline{K_{n}}]) = 3\).
	\end{cor}
	
	\begin{cor}\label{cyclewithemptyeven}
		For \(n\) even and \(m \equiv 0\ (\bmod\ 3) \), \(\ld(C_{m}[\overline{K_{n}}]) = 3\).
	\end{cor}
	
	\begin{theorem}
		Let \(n\) be odd and \(G\) be a \(2r\)-regular graph of order \(m\) with chromatic number \(3\) such that the vertices can be colored as follows:
		\begin{enumerate}
			\item number of vertices having any particular color is odd,
			\item For every vertex, \(r\) of its neighbors have one color while the remaining \(r\)-neighbors have another color.
		\end{enumerate}
		Then \(\chi_{ld}(G[\overline{K_{n}}]) = 3\).
	\end{theorem}
	\begin{proof}
 Let \( V(G) = \{v_1, \dots, v_m\}\) be the vertex set of $G$, where  \(\{v_1, \dots, v_k\}\), \(\{v_{k+1}, \dots, v_{k+s}\}\), \(\{v_{k+s+1}, \dots, v_{k+s+t}\}\) are the color partitions,  with $k+s+t=m$. Note that each of \(k, s\) and \(t\) is atleast \( r\). Further, let $V(\overline{K_{n}})=\{x_{1},x_{2},\dots, x_{n}\}$ be the vertex set for $\overline{K_{n}}$. For $1\leq j \leq m$ and $1\leq i \leq n$, let $v_{i}^j$ be the vertices of $G[\overline{K_{n}}]$ that replace vertices $v_{j}$ of $G$. As $k$, $s$, $t$ and $n$ are all odd, we have three magic rectangles $A=(a_{i,j})$, $B=(b_{i,j})$ and $C=(c_{i,j})$ of size $n\times k$, $n\times s$ and $n\times t$ respectively.
		Using these three magic rectangles, we define a labeling
		$f\colon V(G[\overline{K_{n}}])\to \{1,2,\dots,nm\}$ by
		\begin{equation*} f(v_{i}^j)=
			\begin{cases}
				a_{i,j} &\text{for $1\leq i \leq n$ and $1\leq j \leq k$,}\\
				b_{i,j}+nk &\text{for $1\leq i \leq n$ and $k+1\leq j \leq k+s$,}\\
				c_{i,j}+n(k+s) &\text{for $1\leq i \leq n$ and $k+s+1\leq j \leq m$.}
				
			\end{cases}
		\end{equation*}
		For the weight of vertices, we have, for $1\leq i \leq n$,
		\begin{equation*}
			w(v_{i}^j)= \sum_{v_p\in N(v_j)}\sum_{i=1}^{n}f(v_i^p)=
			\begin{cases}
				\frac{rn(ns+1)}{2} + \frac{rn(nt + 1)}{2} + rn^2(2k+s)  &\text{for $1\leq j \leq k$,}\\\\
				\frac{rn(nk+1)}{2} + \frac{rn(nt + 1)}{2} + rn^2(k+s)  &\text{for  $k+1\leq j \leq k+s$,}\\\\
				\frac{rn(nk+1)}{2} + \frac{rn(ns + 1)}{2} + rn^2k  &\text{for $k+s+1\leq j \leq m$.}
				
			\end{cases}
		\end{equation*}
		
		All three weights are distinct and therefore, $f$ is a local distance antimagic labeling of $G[\overline{K_{n}}]$. Hence $\chi_{ld}(G[\overline{K_{n}}])\leq 3$. As $\chi (G[\overline{K_{n}}])=3$, the equality follows.
	\end{proof}
	\begin{cor}
		For odd integers \(n\) and $m$, \(\ld(K_{m,m,m}[\overline{K_{n}}]) = 3\).
	\end{cor}
	
	\begin{cor}\label{cyclewithemptyodd}
		For integers \(n\) odd and \(m \equiv 0\ (\bmod\ 3) \), \(\ld(C_{m}[\overline{K_{n}}]) = 3\).
	\end{cor}
	In Corollary \ref{cyclewithemptyodd}, we see that if $n$ is odd and  \(m \equiv 0\ (\bmod\ 3) \), then \(\ld(C_{m}[\overline{K_{n}}]) = 3\). In the following Theorem, we generalize this result for any form of odd integer $m$ and any integer $n$.
	\begin{theorem}
		For integers $m\geq 5$ odd and $n\geq 3$, $\chi_{ld}(C_{m}[\overline{K_{n}}])=3$. 
	\end{theorem}
	\begin{proof}
		Let $V(C_{m})=\{v_{1},v_{2},\dots, v_{m}\}$ and $V(\overline{K_{n}})=\{x_{1},x_{2},\dots, x_{n}\}$ be the vertex sets of $C_{m}$ and $\overline{K_{n}}$ respectively. For $1\leq j \leq m$ and $1\leq i \leq n$, let $x_{i}^j$ be the vertices of $C_{m}[\overline{K_{n}}]$.
		We prove the result in two main cases depending upon the parity of $n$.
  \begin{description}
      \item[Case 1:]  $n$ is even. \par
      We study this case in three more cases depending on the nature of $m$.
      \begin{description}
          \item[Case 1.1:] $m\equiv  1\ (\bmod\ 6)$.\par
          Using the definition of Matrix \ref{matrixA}, we construct three matrices $A=(a_{i,j})$, $B=(b_{i,j})$ and $C=(c_{i,j})$. Matrix $A$ is of  size $n \times (\frac{m+2}{3})$, while the matrices $B$ and $C$ are of size $n \times(\frac{m-1}{3})$. We use these three matrices to label $C_{m}[\overline{K_{n}}]$. Define $f\colon V(C_{m}[\overline{K_{n}}])\to \{1,2,\dots,mn\}$ by 
		\begin{equation*}
			f(x_{i}^j)=
			\begin{cases}
				a_{i,(\frac{j+1}{2})} &\text{for $j\equiv 1\ (\bmod\ 3)$ and $1\leq i \leq n$,}\\
				b_{i,(\frac{j+1}{3})}+\frac{n(m+2)}{3} &\text{for $j\equiv 2\ (\bmod\ 3)$ and $1\leq i \leq n$,}\\
				c_{i,(\frac{j}{3})}+\frac{n(2m+1)}{3} &\text{for $j\equiv 0\ (\bmod\ 3)$ and $1\leq i \leq n$.}
			\end{cases}
		\end{equation*}\\
		 For the weight of the vertices, we have, for $1\leq i \leq n$ and $1\leq j \leq m$, $$w(x_{i}^j)=\sum_{i=1}^n\big(f(x_i^{j-1})+f(x_i^{j+1})\big),$$ where $j-1$ and $j+1$ are taken modulo $m$ and therefore,\\
		\begin{equation*}
			w(x_{i}^j)=
			\begin{cases}
				\frac{n}{2}\big(\frac{n(m+2)}{3}+1\big)+\frac{n^{2}(m+1)}{3}\\ +\frac{n}{2}\big(\frac{n(m-1)}{3}+1\big) &\text{ for $j=1$, $j\equiv 0\ (\bmod\ 3)$,}\\\\
				\frac{n}{2}\big(\frac{n(m+2)}{3}+1\big)+\frac{n^{2}(2m+1)}{3}\\ +\frac{n}{2}\big(\frac{n(m-1)}{3}+1\big) &\text{ for $j\equiv 2\ (\bmod\ 3)$,}\\\\
				n\big(\frac{n(m-1)}{3}+1\big)+	\frac{n^{2}(m+1)}{3}\\ +\frac{n^{2}(2m+1)}{3} &\text{ for $j\geq 2$, $j\equiv 1\ (\bmod\ 3)$.}
			\end{cases}
		\end{equation*}\\
   We see that all three weights are distinct.  As $w(x_i^j)\neq w(x_i^{j+1})$, for all $1\leq i \leq n$ and $1\leq j \leq m$, where $j$ and  $j+1$ are taken modulo $m$, $f$ is a local distance antimagic labeling of $C_m[\overline{K_{n}}]$ that assigns three distinct weights to its vertices.
  \item[Case 1.2:] $m\equiv3\ (\bmod\ 6).$ \par
  The result for this case follows from Corollary \ref{cyclewithemptyeven}.
      
      \item[Case 1.3:] $m\equiv5\ (\bmod\ 6).$\par
      When $m\equiv 5\ (\bmod\ 6)$ then either $m\equiv 1\ (\bmod\ 4)$ or $m\equiv 3\ (\bmod\  4)$. So we shall deal with this case in two special cases.
      \begin{description}
          \item[Case 1.3.1:] $m\equiv 1\ (\bmod\ 4).$\par
          Using the definition of Matrix \ref{matrixA}, we construct three matrices $A=(a_{i,j})$, $B=(b_{i,j})$ and $C=(c_{i,j})$. Matrix $A$ is of  size $n \times (\frac{m+1}{2})$, while the matrices $B$ and $C$ are of size $n \times(\frac{m-1}{4})$. We use these three matrices to label $C_{m}[\overline{K_{n}}]$. Define $f\colon V(C_{m}[\overline{K_{n}}])\to \{1,2,\dots,mn\}$ by 
		\begin{equation*}
			f(x_{i}^j)=
			\begin{cases}
				b_{i,(\frac{j+3}{4})}+\frac{n(m+1)}{2} &\text{for $j\equiv 1\ (\bmod\ 4)$, $j\not =m$ and $1\leq i \leq n$,}\\
				c_{i,(\frac{j+2}{4})}+\frac{n(3m+1)}{4} &\text{for $j\equiv 2\ (\bmod\ 4)$ and $1\leq i \leq n$,}\\
				a_{i,(\frac{j+1}{4})} &\text{for $j\equiv 3\ (\bmod\ 4)$ and $1\leq i \leq n$,}\\
				a_{i,(\frac{m-1}{4}+\frac{j}{4})} &\text{for $j\equiv 0\ (\bmod\ 4)$ and $1\leq i \leq n$,}\\
				a_{i, (\frac{m+1}{2})} &\text{for $j=m$ and $1\leq i \leq n$.}
			\end{cases}
		\end{equation*}
		For the weight of the vertices, we have, for $1\leq i \leq n$ and $1\leq j \leq m,$  $$w(x_{i}^j)=\sum_{i=1}^n\big(f(x_i^{j-1})+f(x_i^{j+1})\big),$$ where $j-1$ and $j+1$ are taken modulo $m$ and therefore,\\
		\begin{align*}
			w(x_{i}^j)&=
			\begin{cases}
				\frac{n}{2}\big(\frac{n(m-1)}{4}+1\big)+ \frac{n^{2}(3m+1)}{4}\\ +\frac{n}{2}\big(\frac{n(m+1)}{2}+1\big) &\text{ for $j\equiv 1\ (\bmod\ 2)$, $j\not=m$,}\\\\
				\frac{n}{2}\big(\frac{n(m-1)}{4}+1\big)+ \frac{n^{2}(m+1)}{2} \\ +\frac{n}{2}\big(\frac{n(m+1)}{2}+1\big) &\text{ for $j\equiv 0\ (\bmod\ 2)$, $j=m$, $j\not = m-1$,}\\\\
				n\big(\frac{n(m+1)}{2}+1\big) &\text{ for $j=m-1$.}
			\end{cases}
		\end{align*}\\
  We see that all three weights are distinct.  As $w(x_i^j)\neq w(x_i^{j+1})$, for all $1\leq i \leq n$ and $1\leq j \leq m$, where $j$ and  $j+1$ are taken modulo $m$, $f$ is a local distance antimagic labeling of $C_m[\overline{K_{n}}]$ that assigns three distinct weights to its vertices.\\
		
  \item[Case 1.3.2:] $m\equiv 3\ (\bmod\ 4).$\par
  Using the definition of Matrix \ref{matrixA}, we construct three matrices $A=(a_{i,j})$, $B=(b_{i,j})$ and $C=(c_{i,j})$. Matrix $A$ is of  size $n \times (\frac{m-1}{2})$, while the matrices $B$ and $C$ are of size $n \times(\frac{m+1}{4})$. We use these three matrices to label $C_{m}[\overline{K_{n}}]$. Define $f\colon V(C_{m}[\overline{K_{n}}])\to \{1,2,\dots,mn\}$ by 
		\begin{equation*}
			f(x_{i}^j)=
			\begin{cases}
				b_{i,(\frac{j+3}{4})}+\frac{n(m-1)}{2} &\text{for $j\equiv 1\ (\bmod\ 4)$, and $1\leq i \leq n$,}\\
				c_{i,(\frac{j+2}{4})}+\frac{n(3m-1)}{4} &\text{for $j\equiv 2\ (\bmod\ 4)$ and $1\leq i \leq n$,}\\
				a_{i,(\frac{j+1}{4})} &\text{for $j\equiv 3\ (\bmod\ 4)$ and $1\leq i \leq n$,}\\
				a_{i,(\frac{m+1}{4}+\frac{j}{4})} &\text{for $j\equiv 0\ (\bmod\ 4)$ and $1\leq i \leq n$.}\\
			\end{cases}
		\end{equation*}\\
		For the weight of the vertices, we have, for $1\leq i \leq n$ and $1\leq j \leq m,$  $$w(x_{i}^j)=\sum_{i=1}^n\big(f(x_i^{j-1})+f(x_i^{j+1})\big),$$ where $j-1$ and $j+1$ are taken modulo $m$ and therefore,\\
		\begin{equation*}
			w(x_{i}^j)=
			\begin{cases}
				\frac{n}{2}\big(\frac{n(m+1)}{4}+1\big)+	\frac{n}{2}\big(\frac{n(m-1)}{2}+1\big)+\\ \frac{n^{2}(3m-1)}{4} &\text{ for $j\equiv 1\ (\bmod\ 2)$, $j\not=m$,}\\\\
				\frac{n}{2}\big(\frac{n(m+1)}{4}+1\big)+	\frac{n}{2}\big(\frac{n(m-1)}{2}+1\big)+\\ \frac{n^{2}(m-1)}{2} &\text{ for $j\equiv 0\ (\bmod\ 2)$,}\\\\
				n\big(\frac{n(m+1)}{4}+1\big)+\frac{n^{2}(3m-1)}{4}+\frac{n^{2}(m-1)}{2}&\text{ for $j=m$.}
			\end{cases}
		\end{equation*}\\
		 We see that all three weights are distinct.  As $w(x_i^j)\neq w(x_i^{j+1})$, for all $1\leq i \leq n$ and $1\leq j \leq m$, where $j$ and  $j+1$ are taken modulo $m$, $f$ is a local distance antimagic labeling of $C_m[\overline{K_{n}}]$ that assigns three distinct weights to its vertices.\\
  \end{description}
      \end{description}
      \item[Case 2:] $n$ is odd. \par
      We shall prove the result in three cases depending on the nature of $m$.
      \begin{description}
       \item[Case 2.1:] $m\equiv  1\ (\bmod\  6)$. \par
      As $m\equiv  1\ (\bmod\ 6)$, \big($\frac{m+2}{3}\big)$ is an odd positive integer. So we have a magic rectangle $A$ with entries $(a_{i,j})$ of size $n\times (\frac{m+2}{3})$, having constant column sum as $\frac{n^{2}(m+2)}{6}+\frac{n}{2}$. Further using the definition of Matrix \ref{matrixB} and Matrix \ref{matrixC} given above, we construct two matrices $B=(b_{i,j})$ and $C=(c_{i,j})$ respectively of size $n\times (\frac{m-1}{3})$.  We use these constructions to label $C_{m}[\overline{K_{n}}]$. Define $f\colon V(C_{m}[\overline{K_{n}}])\to \{1,2,\dots,mn\}$ by 
		\begin{equation*}
			f(x_{i}^j)=
			\begin{cases}
				a_{i,(\frac{j+1}{2})} &\text{for $j\equiv 1\ (\bmod\ 3)$ and $1\leq i \leq n$,}\\
				b_{i,(\frac{j+1}{3})}+\frac{n(m+2)}{3} &\text{for $j\equiv 2\ (\bmod\ 3)$ and $1\leq i \leq n$,}\\
				c_{i,(\frac{j}{3})}+\frac{n(2m+1)}{3} &\text{for $j\equiv 0\ (\bmod\ 3)$ and $1\leq i \leq n$.}
			\end{cases}
		\end{equation*}\\
		For the weight of the vertices, we have, for $1\leq i \leq n$ and $1\leq j \leq m,$  $$w(x_{i}^j)=\sum_{i=1}^n\big(f(x_i^{j-1})+f(x_i^{j+1})\big),$$ where $j-1$ and $j+1$ are taken modulo $m$ and therefore,\\
		\begin{equation*}
			w(x_{i}^j)=
			\begin{cases}
				\frac{n}{2}\big(\frac{n(m+2)}{3}+1\big)+(\frac{m-1}{3})n^{2}-4(\frac{m-1}{3})+\\ n^{2}(\frac{m+1}{3})+2&\text{ for $j=1$, $j\equiv 0\ (\bmod\ 3),$}\\\\
				\frac{n}{2}\big(\frac{n(m+2)}{3}+1\big)+(\frac{m-1}{3})n^{2}+4(\frac{m-1}{3})+\\ n^{2}(\frac{2m+1}{3})+n-2 &\text{ for $j\equiv 2\ (\bmod\ 3),$}\\\\
				2(\frac{m-1}{3})n^{2}+n^{2}(m+1)+n &\text{ for $j\geq 2$, $j\equiv 1\ (\bmod\ 3).$}
			\end{cases}
		\end{equation*}\\
		 We see that all three weights are distinct.  As $w(x_i^j)\neq w(x_i^{j+1})$, for all $1\leq i \leq n$ and $1\leq j \leq m$, where $j$ and  $j+1$ are taken modulo $m$, $f$ is a local distance antimagic labeling of $C_m[\overline{K_{n}}]$ that assigns three distinct weights to its vertices.
  \item[Case 2.2:] $m\equiv3\ (\bmod\ 6).$ \par
  The result follows from Corollary \ref{cyclewithemptyodd}.
  \item[Case 2.3:] $m\equiv5\ (\bmod\ 6).$\par
  When $m\equiv 5\ (\bmod\ 6)$ then either $m\equiv 1\ (\bmod \ 4)$ or $m\equiv 3\ (\bmod\ 4)$. So we shall deal with this case in these two special cases.
  \begin{description}
      \item[Case 2.3.1:] $m\equiv 1\ (\bmod\ 4).$ \par
      As $m\equiv  1(\bmod 4)$, \big($\frac{m+1}{2}\big)$ is an odd positive integer. So we have a magic rectangle $A$ with entries $(a_{i,j})$ of size $n\times (\frac{m+1}{2})$, having constant column sum as $\frac{n^{2}(m+1)}{4}+\frac{n}{2}$. Further using the definition of Matrix \ref{matrixB} and Matrix \ref{matrixC} given above, we construct two matrices $B=(b_{i,j})$ and $C=(c_{i,j})$ respectively of size $n\times (\frac{m-1}{4})$.  We use these constructions to label $C_{m}[\overline{K_{n}}]$. Define $f\colon V(C_{m}[\overline{K_{n}}])\to \{1,2,\dots,mn\}$ by 
		\begin{equation*}
			f(x_{i}^j)=
			\begin{cases}
				b_{i,(\frac{j+3}{4})}+\frac{n(m+1)}{2} &\text{for $j\equiv 1\ (\bmod\ 4)$, $j\not =m$ and $1\leq i \leq n$,}\\
				c_{i,(\frac{j+2}{4})}+\frac{n(3m+1)}{4} &\text{for $j\equiv 2\ (\bmod\ 4)$ and $1\leq i \leq n$,}\\
				a_{i,(\frac{j+1}{4})} &\text{for $j\equiv 3\ (\bmod\ 4)$ and $1\leq i \leq n$,}\\
				a_{i,(\frac{m-1}{4}+\frac{j}{4})} &\text{for $j\equiv 0\ (\bmod\ 4)$ and $1\leq i \leq n$,}\\
				a_{i, (\frac{m+1}{2})} &\text{for $j=m$ and $1\leq i \leq n$.}
			\end{cases}
		\end{equation*}\\
		For the weight of the vertices, we have, for $1\leq i \leq n$ and $1\leq j \leq m,$  $$w(x_{i}^j)=\sum_{i=1}^n\big(f(x_i^{j-1})+f(x_i^{j+1})\big),$$ where $j-1$ and $j+1$ are taken modulo $m$ and therefore,\\
		\begin{equation*}
			w(x_{i}^j)=
			\begin{cases}
				\frac{n}{2}\big(\frac{n(m+1)}{2}+1\big)+(\frac{m-1}{4})n^{2}+4(\frac{m-1}{4})+\\ n^{2}(\frac{3m+1}{2})+n-2 &\text{ for $j\equiv 1\ (\bmod\ 2)$, $j\not=m-1$,}\\\\
				\frac{n}{2}\big(\frac{n(m+1)}{2}+1\big)+(\frac{m-1}{4})n^{2}-4(\frac{m-1}{4})+\\ n^{2}(\frac{m+1}{2})+2 &\text{ for $j\equiv 0\ (\bmod\ 2)$,}\\\\
				n\big(\frac{n(m+1)}{2}+1\big) &\text{ for $j=m-1$.}
			\end{cases}
		\end{equation*}\\
		 We see that all three weights are distinct.  As $w(x_i^j)\neq w(x_i^{j+1})$, for all $1\leq i \leq n$ and $1\leq j \leq m$, where $j$ and  $j+1$ are taken modulo $m$, $f$ is a local distance antimagic labeling of $C_m[\overline{K_{n}}]$ that assigns three distinct weights to its vertices.\\
  \item[Case 2.3.2:] $m\equiv 3\ (\bmod\ 4).$\par
  As $m\equiv  3\ (\bmod\ 4)$, \big($\frac{m-1}{2}\big)$ is an odd positive integer. So we have a magic rectangle $A$ with entries $(a_{i,j})$ of size $n\times (\frac{m-1}{2})$, having constant column sum as $\frac{n^{2}(m-1)}{4}+\frac{n}{2}$. Further, using the definition of Matrix \ref{matrixB} and Matrix \ref{matrixC} given above, we construct two matrices $B=(b_{i,j})$ and $C=(c_{i,j})$ respectively of size $n\times (\frac{m+1}{4})$.  We use these constructions to label $C_{m}[\overline{K_{n}}]$. Define $f\colon V(C_{m}[\overline{K_{n}}])\to \{1,2,\dots,mn\}$ by 
		\begin{equation*}
			f(x_{i}^j)=
			\begin{cases}
				b_{i,(\frac{j+3}{4})}+\frac{n(m-1)}{2} &\text{for $j\equiv 1\ (\bmod\ 4)$ and $1\leq i \leq n$,}\\
				c_{i,(\frac{j+2}{4})}+\frac{n(3m-1)}{4} &\text{for $j\equiv 2\ (\bmod\ 4)$ and $1\leq i \leq n$,}\\
				a_{i,(\frac{j+1}{4})} &\text{for $j\equiv 3\ (\bmod\  4)$ and $1\leq i \leq n$,}\\
				a_{i,(\frac{m+1}{4}+\frac{j}{4})} &\text{for $j\equiv 0\ (\bmod\ 4)$ and $1\leq i \leq n$.}\\
			\end{cases}
		\end{equation*}\\
		For the weight of the vertices, we have, for $1\leq i \leq n$ and $1\leq j \leq m,$  $$w(x_{i}^j)=\sum_{i=1}^n\big(f(x_i^{j-1})+f(x_i^{j+1})\big),$$ where $j-1$ and $j+1$ are taken modulo $m$ and therefore,\\
		\begin{equation*}
			w(x_{i}^j)=
			\begin{cases}
				\frac{n}{2}\big(\frac{n(m-1)}{2}+1\big)+(\frac{m+1}{4})n^{2}+4(\frac{m+1}{4})+\\ n^{2}(\frac{3m-1}{4})+n-2 &\text{ for $j\equiv 1\ (\bmod\ 2)$, $j\not=m$,}\\\\
				\frac{n}{2}\big(\frac{n(m-1)}{2}+1\big)+(\frac{m+1}{4})n^{2}-4(\frac{m+1}{4})+\\ n^{2}(\frac{m-1}{2})+2
				&\text{ for $j\equiv 0\ (\bmod\ 2)$,}\\\\
				n\big(\frac{n(m-1)}{2}+1\big)+\frac{n^{2}(5m-3)}{4}+\frac{n^{2}(m+1)}{2}+n&\text{ for $j=m$.}
			\end{cases}
		\end{equation*}\\
		We see that all three weights are distinct.  As $w(x_i^j)\neq w(x_i^{j+1})$, for all $1\leq i \leq n$ and $1\leq j \leq m$, where $j$ and  $j+1$ are taken modulo $m$, $f$ is a local distance antimagic labeling of $C_m[\overline{K_{n}}]$ that assigns three distinct weights to its vertices.
  \end{description}
      \end{description}
       \end{description}
	 In each case, we get three distinct weights and therefore $\chi_{ld}(C_m[\overline{K_{n}}])\leq 3$. As $\chi(C_m[\overline{K_{n}}])= 3$, we have $\chi_{ld}(C_m[\overline{K_{n}}])= 3.$
	\end{proof}

	After studying graphs with chromatic number 3 satisfying $\chi_{ld}(G[\overline{K_{n}}])=3$, we now study the lexicographic product of the complete graph with the complement of the complete graph. 
	\begin{theorem}
 For positive integers $m$ and $n>1$, and a $r$-regular graph $H$ of order $n$, we have $\chi_{ld}(K_m[H])\leq m\cdot \chi_{ld}(H)$.
		\begin{proof}
			Let $V(K_m)=\{x_1,x_2,\dots,x_m\}$ and $V(H)=\{y_1,y_2,\dots,y_n\}$ be the vertex sets of $K_m$ and $H$ respectively. For $1\leq j \leq m$, let $y_{i}^j$ be the vertices  of $K_m[H]$ that correspond with vertices $y_i$ of $H$, where $1\leq i \leq n$. Let $f$ be a local distance antimagic labeling of $H$ that assigns $\chi_{ld}(H)$ distinct weights to vertices of $H$. Using $f$, we define a labeling $g$ for vertices of $K_m[H]$ by $$g(y_{i}^j)=f(y_i)+(j-1)n\quad \text{where $1\leq i \leq n$ and $1\leq j \leq m$}.$$ 
			For the weight of the vertices, we have,
			\begin{align*}
				w(y_{i}^j)&=\sum_{p=1, p\neq j}^m\sum_{i=1}^n f(y_i^p)+\sum_{y_k\in N(y_i)}f(y^j_k),\\   &=\frac{n(n+1)(m-1)}{2}+\bigg(\frac{(m-1)m}{2}-(j-1)\bigg)n^2+ w_f(y_i)+rn(j-1),\\
				&=\frac{n(m-1)(nm+n+1))}{2}+w_f(y_i)-(j-1)n(n-r),
			\end{align*}
			where, $w_f(y_i)$ represents the weight of the vertex $y_i$ in $H$, under labeling $f$.\\
			Observe that for fixed $j\in \{1,2,\dots,m\}$, we get $\chi_{ld}(H)$ distinct weights. Therefore we get atmost $m\cdot \chi_{ld}(H)$ distinct weights.
		\end{proof}
	\end{theorem}
	\begin{cor}
		For positive integers $m$ and $n>1$, $\chi_{ld}(K_m[\overline{K_{n}}])=m$.
	\end{cor}
	Above, we have presented several regular non-bipartite graphs $G$, where the equality $\chi_{ld}(G[\overline{K_{n}}])=\chi(G)$  holds. We now present some non-regular, non-bipartite graphs with chromatic number 3, where the equality holds.
	\begin{theorem}\label{lexijointhm}
		Let $n>1$ and $G$ be a graph of order $m$ having maximum degree $\Delta$. Let $H=G+K_1$. If 
		\begin{equation}\label{lexijoineq}
			\Delta n (mn+n)-\frac{\Delta n (\Delta n-1)}{2}< 2mn^2+\frac{(m-1)n(mn+n+1)}{2},
		\end{equation}
then $\chi_{ld}(H[\overline{K_{n}}])\leq \chi_{ld}(G[\overline{K_{n}}])+1$. 
\begin{proof}
	As lexicographic product distributes over join, we have $H[\overline{K_{n}}]=G[\overline{K_{n}}]+K_1[\overline{K_{n}}]$. Note that $G[\overline{K_{n}}]$ is a graph of order $mn$, with maximum degree $\Delta n$, while $K_1[\overline{K_{n}}]$ is a graph of order $n$ with all the vertices of degree zero. As Equation \ref{lexijoineq} holds, the Equation \ref{equation0} in Theorem \ref{jointhm1} holds and therefore we have, $\chi_{ld}(H[\overline{K_{n}}])\leq \chi_{ld}(G[\overline{K_{n}}])+1$.   
\end{proof}
	\end{theorem}
		\begin{cor}
		Let $n>1$ and $G$ be a $r$-regular bipartite graph of order $m$. Let $H=G+K_1$. If 
		\begin{equation}\label{lexijoinK1}
			 rn (mn+n)-\frac{ rn (r n-1)}{2}< 2mn^2+\frac{(m-1)n(mn+n+1)}{2},
		\end{equation}
		then $\chi_{ld}(H[\overline{K_{n}}])=3$. 
		\begin{proof}
		The result follows from Theorem \ref{lexijointhm}, Theorem \ref{bipartitelexi1}. 
		\end{proof}
	\end{cor}
	
	The wheel graph $W_m$ is a graph of order $m+1$, formed by joining all the vertices of cycle $C_m$ to a new vertex called the central vertex, i.e. $W_m=C_m+K_1$.
	\begin{cor}
		For positive integers $n>1$ and $m>3$ even, $\chi_{ld}(W_m[\overline{K_{n}}])=3$.
	\end{cor} 
	The friendship graph $F_m$ is a graph of order $2m+1$, formed by joining all the vertices of  $m$ copies of $K_2$ to a new vertex called the central vertex, i.e., $F_m=mK_2+K_1$.
	\begin{cor}
		For positive integers $n>1$ and $m>1$, $\chi_{ld}(F_m[\overline{K_{n}}])=3$.
	\end{cor}
	We have explored several regular graphs $G$ with chromatic number 3, for which $\chi_{ld}(G[\overline{K_{n}}])=3$. We have also explored some non-regular graphs $G$ with chromatic number 3, for which  $\chi_{ld}(G[\overline{K_{n}}])=3$. We now study some non-regular graphs with chromatic number 3, for which $\chi_{ld}(G[\overline{K_{n}}])\not=3$. 
	\begin{theorem}
		Let $n>1$ and $G$ be a bipartite graph having two vertices $x$ and $y$ belonging to the same partite set such that $N(x)\subset N(y)$. If $H=G+K_1$, $\chi_{ld}(H[\overline{K_{n}}])\geq 4$. 
		\begin{proof}
			Since the chromatic number of $H$ is $3$, we have $\chi_{ld}(H[\overline{K_{n}}])\geq 3$. Let $v$ be the lone vertex that is joined to all other vertices of $G$ in $H$. For $1\leq i \leq n$, let $x_i$, $y_i$ and $v_i$ be the vertices of $H[\overline{K_{n}}]$, that correspond with vertices $x$, $y$ and $v$ respectively of $H$.  As $N(x)\subset N(y)$,  by the construction of $G[\overline{K_{n}}]$ we have, $N(x_i)\subset N(y_i)$ for all $1\leq i \leq n$. Consider the vertex $v_1$ of $H[\overline{K_{n}}]$. It is adjacent to vertices $x_1$ and $y_1$.  So, the weight of $v_1$ is distinct from the weights of $x_1$ and $y_1$. If $\chi_{ld}(H[\overline{K_{n}}])=3$, then since $x$ and $y$ belong to the same partite set, $w(x_{1})=w(y_{1})$. But as $N(x_1)\subset N(y_1)$, this is not possible. Therefore $\chi_{ld}(H[\overline{K_{n}}])\geq 4$.
		\end{proof}
	\end{theorem}
		The fan graph $T_m$ is a graph of order $m+1$, formed by joining all the vertices of path $P_m$ to a new vertex called the central vertex, i.e. $T_m=P_m+K_1$.
		\begin{cor}
			For integers $n>1$, $m>3$, $\chi_{ld}(T_m[\overline{K_{n}}])\geq 4$.
		\end{cor}
		\begin{cor}
			For integers $n>1$, $m>3$, $\chi_{ld}(T_m[\overline{K_{n}}])\leq 5$.
		\end{cor}
		\begin{proof}
			The result follows from Theorem \ref{pathcomposition1}, Theorem \ref{pathcomposition2}, and Theorem \ref{lexijointhm}.
		\end{proof}
  \section{Conclusion}
 In this paper, we provided two examples of classes of graphs $G$ of order $n$ having $\chi_{ld}(G)=p,$ where $2\leq p \leq n.$
Further in Section 3, we studied the local distance antimagic chromatic number for the join of graphs and presented examples of classes of graphs $G$ and $H$ for which $\chi_{ld}(G+H)=\chi_{ld}(G)+\chi_{ld}(H).$\\
 Section 4 studied the lexicographic product of bipartite graphs with the complement of the complete graph. We showed that the local distance antimagic chromatic number of the lexicographic product of regular bipartite graphs with the complement of the complete graph is 2. Many classes of non-regular bipartite graphs were presented for which $\chi_{ld}(G[\overline{K_{n}}])=2$. Later in the section, we presented examples of classes of graphs with chromatic number 3 for which $\chi_{ld}(G[\overline{K_{n}}])=3$.\\\\
 The following problems naturally arise:
 \begin{problem}
     Characterise graphs $G$ and $H$ for which $\chi_{ld}(G+H)=\chi_{ld}(G)+\chi_{ld}(H).$
 \end{problem}
 \begin{problem}
     Characterise non-regular bipartite graphs $G$ for which $\chi_{ld}(G[\overline{K_{n}}])=2$.
 \end{problem}
 \begin{problem}
     Characterise graphs $G$ with chromatic number 3 for which $\chi_{ld}(G[\overline{K_{n}}])=3$.
 \end{problem}

\end{document}